\documentclass[12pt]{amsart}

\usepackage[utf8]{inputenc}    
\usepackage[T1]{fontenc}       
\usepackage{amsmath, amssymb, amsthm} 
\usepackage{graphicx}          
\usepackage{hyperref}         
\usepackage{geometry}         
\usepackage{enumitem}          
\usepackage{caption}           
\usepackage{xcolor}           
\usepackage{cleveref}

\newtheorem{theorem}{Theorem}[section]
\newtheorem{lemma}[theorem]{Lemma}
\newtheorem{proposition}[theorem]{Proposition}

\theoremstyle{definition}

\theoremstyle{remark}
\newtheorem{remark}[theorem]{Remark}

\title{New upper bounds for average kissing numbers}

\author{Mark Bockhaus}

\address{Mark Bockhaus, School of Mathematical and Statistical Sciences, The University of Texas Rio Grande Valley, Brownsville, TX 78520, USA}
\email{mark.bockhaus01@utrgv.edu}

\author{Alexey Glazyrin}
\address{Alexey Glazyrin, School of Mathematical and Statistical Sciences, The University of Texas Rio Grande Valley, Brownsville, TX 78520, USA}
\email{alexey.glazyrin@utrgv.edu}

\date{}

\begin{document}

\maketitle

\begin{abstract}
A packing of balls in Euclidean space is a set of balls with non-overlapping interiors. For a finite packing of balls, its average contact number is the average number of balls tangent to a ball in this packing, that is, twice the total number of tangencies divided by the number of balls. The average kissing number of the space is defined as the supremum of average contact numbers over all finite ball packings. In this paper, we present a new semidefinite programming approach to average kissing numbers and improve upon known upper bounds in all dimensions \(3\) through \(30\). In particular, we provide the first improvements to the upper bounds in dimensions \(10\) through \(30\) obtained from kissing numbers. We also fully solve the average kissing number problem in dimensions \(8\) and \(24\) by showing that the average kissing numbers coincide with the kissing numbers, 240 and 196560, respectively. These values are limits of average contact numbers of finite subpackings from the \(E_8\) and Leech lattices.
\end{abstract}

\section{Introduction}\label{sect:intro}
Over three hundred years ago, the kissing number problem was born during a discussion between Isaac Newton and David Gregory \cite{Zong:1999}, \cite{Pfender:2004}, \cite{Boyvalenkov:2012}. The problem in question was: Given a unit three-dimensional sphere, how many non-overlapping unit spheres can simultaneously touch this initial sphere? While the problem was initially conceived in three dimensions, the idea can easily be extended to any dimension.  Thus, the kissing number problem can be generalized as: Given an \(n\)-dimensional unit ball, how many non-overlapping \(n\)-dimensional unit balls can simultaneously touch this initial ball?  This problem has been the subject of significant research, with the kissing number, \(\tau_n\), being known exactly only in dimensions \(1, 2, 3, 4, 8\) and \(24\) \cite{Schutte:1953}, \cite{Musin:2006}, \cite{Maehara:2007}, \cite{Glazyrin:2023}, \cite{Musin:2008}, \cite{Bachoc:2008}, \cite{Levenshtein:1979}, \cite{Odlyzko:1979}. 

While the kissing number problem requires all balls to be congruent, it is also of interest to consider what may happen when we allow balls of different radii. Of course, for any particular ball we can choose an arbitrarily large number of small balls tangent to it, but the average number of tangencies cannot grow indefinitely. Thus it would be interesting to determine how large this average number can be for a packing of balls. We clarify that a \textit{ball packing} is a set of \(n\)-dimensional balls (of any radii) with pairwise disjoint interiors. We now formalize these notions. Let \(P\) be a finite packing of \(n\)-dimensional balls in \(\mathbb{R}^n\). Let \(E\) be the set of unordered pairs of balls that are tangent to each other or, in other words, the set of edges in the contact graph of the packing. Then the average contact number of the packing is given by: 
    \begin{align*}
        k(P) =\frac{2|E|}{|P|}
    \end{align*}

While we can determine the average contact number of any packing of balls, as with the kissing number problem, we are interested in maximizing this value. Thus, for each dimension, \(n\), we define the \textit{average kissing number} as:
    \begin{align*}
        \kappa_n=\text{sup} \{k(P):P \text{ is a finite packing of balls in }\mathbb{R}^n \}
    \end{align*}

At this point, we also note that $\kappa_n$ is the same for all $n$-dimensional spaces of constant curvature, that is, for Euclidean spaces, spheres, and hyperbolic spaces, and does not depend on the curvature. Moreover, M\"{o}bius transformations also preserve the contact number of a ball packing, if the relative position of the balls is unchanged. All this makes the problem of finding the average kissing number, on the one hand, more fundamental than the kissing number problem and, on the other hand, more difficult.

Finding $\kappa_2$ is straightforward. The contact graph of any packing of disks is planar. The average degree of a planar graph is not larger than 6 by Euler's formula. This bound is attained as a limit for finite configurations from the hexagonal packing of equal disks.

The same idea for finding lower bounds works in all dimensions. In particular, the lattice kissing number, \(\tau_n^*\), which is the largest kissing number for all lattices in \(\mathbb{R}^n\), provides a lower bound for the average kissing number.  In some dimensions, non-lattice periodic packings of unit balls provide a larger average contact number; however, in many cases, the lattice kissing number remains the best known lower bound. 

In terms of upper bounds, a simple argument shows that in any dimension \(\kappa_n \le 2\tau_n\). Indeed, we can order the balls of a packing in increasing order by size. If some balls are of the same size, we choose an arbitrary order between them. Then for any ball in a packing, it has no more than $\tau_n$ neighbors occurring later in the chosen order. Counting oriented edges (from small to large) in a contact graph we immediately get the required bound.

The problem of determining average kissing numbers was first considered by Kuperberg and Schramm \cite{Kuperberg:1994}. They found the first nontrivial bounds for \(\kappa_3\). Their lower bound \(\kappa_3\geq\frac{666}{53}\approx12.566\), in contrast to the lattice and periodic packing bounds mentioned above, was based on the 600-cell and was later improved by Eppstein, Kuperberg, and Ziegler \cite{Eppstein:2003} who showed that \(\kappa_3 \ge
\frac{7656}{607}\approx12.6129\) using an even more intricate modification of the ball packing defined by the 600-cell. 

Kuperberg and Schramm also proved the first nontrivial upper bound in dimension 3: \(\kappa_3\leq 8+4\sqrt{3}\approx 14.9282\). In Section \ref{sect:shell} we briefly review their approach and its improvements and extensions that lead to better upper bounds in dimensions 3 through 9.

Simple upper bounds obtained from \(2\tau_n\) by using upper bounds for kissing numbers were the best known bounds in dimensions 10 and higher. The comparison bounds use the upper bounds for kissing numbers compiled by Cohn \cite{Cohn}. These include bounds obtained from the semidefinite programming method of Bachoc and Vallentin \cite{Bachoc:2008} and its computational refinements by Machado and Oliveira \cite{Machado:2018} and Leijenhorst and de Laat \cite{Leijenhorst:2024}. This paper improves upon the known bounds in dimensions 3 through 9, and provides the first improvements to the simple upper bound in dimensions larger than 9. We present new bounds in dimensions up to 30, but the machinery certainly works in larger dimensions as well, and 30 was chosen just as a somewhat arbitrary stopping point. We also fully solve the average kissing number problem in dimensions 8 and 24. In these dimensions, the answer is the same as the regular kissing number. In dimension 8, we prove that the average kissing number is not larger than 240. This bound is reached as a lattice kissing number of the lattice \(E_8\). In dimension 24, the average kissing number is not larger than 196560 and it is reached as a lattice kissing number of the Leech lattice. Our new upper bounds for average kissing numbers are presented in \Cref{tab:NEW}.  

\begin{table} 
    \begin{center}
    \begin{tabular}{||c c c c ||} 
        \hline
        Dimension & Lower Bound & Previous Upper Bound & New Upper Bound \\ [0.5ex] 
        \hline\hline
                 3 & 12.612 &  13.606 &      13.327 \\
         4 &     24 &  27.439 &      25.786 \\
         5 &     40 &  64.022 &      46.660 \\
         6 &     72 & 121.105 &      82.646 \\
         7 &    126 & 223.144 &     140.384 \\
         8 &    240 & 408.386 &         240 \\
         9 &    272 & 722.629 &     380.244 \\
        10 &    336 &    1106 &     596.184 \\
        11 &    438 &    1736 &     916.228 \\
        12 &    756 &    2710 &    1418.098 \\
        13 &    918 &    4128 &    2233.684 \\
        14 &   1422 &    6348 &    3492.338 \\
        15 &   2340 &    9706 &    5431.321 \\
        16 &   4320 &   14640 &    8314.473 \\
        17 &   5346 &   21956 &   12219.420 \\
        18 &   7398 &   32812 &   17878.342 \\
        19 &  10668 &   48834 &   25902.460 \\
        20 &  17400 &   72390 &   37976.885 \\
        21 &  27720 &  107048 &   56854.921 \\
        22 &  49896 &  161620 &   86543.475 \\
        23 &  93150 &  244702 &  128104.059 \\
        24 & 196560 &  393120 &      196560 \\
        25 & 196656 &  530012 &  278379.872 \\
        26 & 196848 &  735550 &  396999.509 \\
        27 & 197142 & 1044424 &  560721.861 \\
        28 & 197736 & 1504584 &  790758.685 \\
        29 & 198506 & 2151982 & 1130937.482 \\
        30 & 200046 & 3075414 & 1654023.607 \\ [1ex] 
        \hline
    \end{tabular}
    \caption[New upper bounds for the average kissing number]{New upper bounds for the average kissing number in dimensions \(3\) through \(30\).  We compare with the known lower bounds and the previous upper bounds.  Previous upper bounds in dimensions \(3\) through \(9\) come from Dostert, Kolpakov, and Oliveira \cite{Dostert:2022}. The lower bound in dimension \(3\) comes from Eppstein, Kuperberg, and Ziegler \cite{Eppstein:2003}, and lower bounds in dimensions \(4\) through \(30\) come from lattice kissing number bounds from Conway and Sloane \cite[Table 1.2]{Conway:1999}. Previous upper bounds in dimensions \(10\) through \(30\) come from twice the kissing number upper bounds \cite{Cohn}.} \label{tab:NEW}
    \end{center}
\end{table}  

We note that all upper bounds so far (and our paper is not an exception) use only the local information about the packing; that is, the packing condition can be relaxed to "all balls tangent to the same ball do not overlap". It would be interesting to find arguments that genuinely use the non-overlapping condition akin to how it is done for the sphere packing problem \cite{Gorbachev:2000}, \cite{Cohn:2003}, \cite{Viazovska:2017}, \cite{Cohn:2017}. \cite[Theorem 4.1]{Dostert:2022} extends the framework of \cite{Cohn:2003} and \cite{deLaat:2014} to the problem of average kissing numbers but it works only when radii are taken from a closed interval or a finite set and, in practice, its bounds do not seem to improve bounds produced by local constraints.

The paper is structured as follows. In Section \ref{sect:shell}, we describe the shell approach of Kuperberg and Schramm and all the recent upper bounds stemming from it. In Section \ref{sect:LP} we describe the linear programming bound for regular kissing numbers. Our semidefinite program essentially contains the linear program of Delsarte as its diagonal slice. In Section \ref{sect:new} we develop a new approach based on auxiliary functions and explain why it generalizes previous attempts to solve the problem. In Section \ref{sect:SDP} we show how the new approach can be implemented as a semidefinite program. In Section \ref{sect:exact} we explain how we found exact polynomial certificates solving the average kissing number problem in dimensions 8 and 24. Finally, in Section \ref{sect:opt} we show how semidefinite programs can be extended to incorporate optimization over auxiliary functions as well and use it to find upper bounds in other dimensions. 

\section{Upper bounds via shell packings}\label{sect:shell}

In this section we briefly describe how average kissing numbers can be bounded via densities of auxiliary packings of spherical caps on spherical shells. We consider, without loss of generality, a unit sphere centered at the origin, \(S^{n-1}\). We then consider a spherical shell of radius \(\rho, 1<\rho<3\), centered at the origin. Thus, we denote this spherical shell by \(\rho S^{n-1}\). Let \(C_r\) be a ball of radius \(r>0\) that is tangent to the central ball. We consider the spherical cap formed by the intersection of this spherical shell and the tangent ball, \(\rho S^{n-1} \cap C_r\).  We denote the normalized volume of this spherical cap, formed by \(S_{\rho}(r)=\rho S^{n-1} \cap C_r\), by
\[
A_{n, \rho}(r):=\operatorname{Vol} (S_{\rho}(r))/\operatorname{Vol} (\rho S^{n-1}).
\]

The intersections of the tangent balls and the spherical shell form a spherical cap packing.  We are interested in the proportion of the spherical shell, \(\rho S^{n-1}\), that can be taken up by any given cap packing. For a dimension \(n\) and the radius of the spherical shell, \(\rho\), we denote by \(\text{dens}_n(\rho)\) the supremum over all possible cap packings of the proportion of the spherical shell that is covered by the spherical caps \(S_{\rho}(r)\). We clarify that the supremum is taken over packings induced by non-overlapping balls tangent to the unit ball. The key fact here is that \(A_{n, \rho}(r)+A_{n, \rho}(1/r)\) reaches its minimum at $r=1$. This was shown by Kuperberg and Schramm \cite{Kuperberg:1994} in dimension 3 and then by Glazyrin \cite{Glazyrin:2020} in all dimensions $n\geq 3$.
A simple double counting of these sums for all pairs of tangent balls leads to the following bound. 
\begin{theorem} \label{thm:densityKissingBound}
    Let \(n \ge 3\) and consider any \(1<\rho<3\). Then
    \begin{align*}
        \kappa_n \le \frac{\text{dens}_n(\rho)}{A_{n, \rho}(1)}.
    \end{align*}
\end{theorem}

The upper bound of Kuperberg and Schramm in dimension 3 just uses $\text{dens}_3(\rho)\leq 1$ and takes the optimal $\rho=\sqrt{3}$. Glazyrin used a nontrivial estimate for the density to improve the bound in dimension 3 and a trivial bound $\text{dens}_n(\rho)\leq 1$ to obtain the first bounds improving on $2\tau_n$ in dimensions 4 and 5. Dostert, Kolpakov, and Oliveira \cite{Dostert:2022} used semidefinite programming to obtain bounds on cap-packing density, and in doing so, found the following bounds for the average kissing numbers: \(\kappa_3 \le 13.606\), \(\kappa_4 \le 27.439\), \(\kappa_5 \le 64.022\), \(\kappa_6 \le 121.105\), \(\kappa_7 \le 223.144\), \(\kappa_8 \le 408.386\), and \(\kappa_9 \le 722.629\). This approach could not improve $2\tau_n$ in higher dimensions. Before this paper, these were the best known bounds for the average kissing number.

\section{Linear programming approach}\label{sect:LP}

As we mentioned above, all approaches so far use the local information of the packing, that is, the configuration of balls tangent to the same central ball. This configuration may be viewed as an extension of the regular kissing configuration.  Therefore, it makes sense that we may obtain bounds on the average kissing number by extending the approaches used to bound the kissing number.

As we generalize this extension to obtain our results, we see fit to describe the linear programming approach. It is based on positive semidefiniteness of Gegenbauer polynomials as zonal spherical functions. Throughout the paper, by Gegenbauer polynomials \(P_{k}^{n}\) we mean Jacobi polynomials $P_{k}^{(\alpha,\beta)}$ with equal parameters \(\alpha=\beta=(n-3)/2\), where \(k\) is the degree of the polynomial and the polynomials are normalized so that \(P_k^n(1)=1\).

Using the addition formula by Gegenbauer \cite{Gegenbauer:1893}, Schoenberg proved the following property of Gegenbauer polynomials \cite{Schoenberg:1942}.

\begin{proposition}\cite{Schoenberg:1942}\label{prop:schoenberg}
For any \(k \ge 1\) and for any set of points $X=\{x_1,\ldots, x_m\}\subset S^{n-1}$, the matrix \((P_k^n(x_i\cdot x_j))_{1 \le i,j \le m}\) is positive semidefinite.
\end{proposition}

The linear programming method was developed by Delsarte \cite{Delsarte:1973}. The version below is essentially due to Delsarte, Goethals, and Seidel \cite[Theorem 4.3]{Delsarte:1977}. Odlyzko and Sloane \cite{Odlyzko:1979} and, independently, Levenshtein \cite{Levenshtein:1979} used it to find kissing numbers in dimensions 8 and 24. The best asymptotic upper bounds for $\tau_n$ are also obtained using this approach \cite{Kabatyansky:1978}, \cite{OpenAI:2026}. The proof is extremely short and important for further discussion, so we include it here.

\begin{theorem} \label{thm:linProg}
    Assume \(n \ge 3\).  If \(f(t)\) is a real polynomial that satisfies:
    \begin{enumerate}
        \item \(f(t) \le 0\) for \(-1 \le t \le \frac{1}{2}\), \text{ and} 
        \item the coefficients in the expansion of \(f(t)\) in terms of Jacobi polynomials
        \begin{align*}
            f(t) =\sum_{i=0}^{d} f_i P_{i}^{n}(t)
        \end{align*}
        satisfy \(f_0>0, f_1 \ge 0,\ldots, f_d \ge 0\). 
    \end{enumerate}
    Then \(\tau_n\) is bounded by:
        \begin{align*}
            \tau_n \le \frac{f(1)}{f_0}.
        \end{align*}
\end{theorem}

\begin{proof}
Let \(X=\{x_1,\ldots,x_m\}\) be a set of points in \(S^{n-1}\) such that \(x_i\cdot x_j\leq \frac 1 2\) for \(i\neq j\). On the one hand,
\[
\sum\limits_{i,j}^m f(x_i\cdot x_j)\leq \sum\limits_{i=1}^m f(x_i\cdot x_i) = mf(1)
\]
because \(f(x_i\cdot x_j)\leq 0\) for all \(i\neq j\). On the other hand,
\[
\sum\limits_{i,j}^m f(x_i\cdot x_j)=\sum\limits_{k=0}^d f_i \sum\limits_{i,j}^m P_{k}^{n}(x_i\cdot x_j)\geq f_0  \sum\limits_{i,j}^m P_{0}^{n}(x_i\cdot x_j) = f_0 m^2
\]
because \(\sum\limits_{i,j}^m P_{k}^{n}(x_i\cdot x_j) \geq 0\) by \Cref{prop:schoenberg}. Combining the two inequalities yields the required bound.
\end{proof}

\Cref{thm:linProg} can be turned into an optimization problem, where we minimize \(f(1)\) with the fixed value \(f_0=1\) subject to the two hypotheses. Then, we can write this as a linear program, of which any feasible solution is an upper bound for the kissing number.  

To bound kissing numbers, we are only concerned with interactions between unit balls.  However, for the average kissing number, we want to consider interactions between balls of any radii. We now briefly discuss how \Cref{thm:linProg} can be extended to include packings of balls with different radii.

The key ingredient of the Delsarte method is the nonnegativity of \(\sum\limits_{i,j=1}^m P_k^n(x_i\cdot x_j)\). All points $x_i$ correspond to unit balls tangent to the central ball and it is natural that all pairwise values $P_k^n(x_i\cdot x_j)$ participate in the sum with the same coefficient. However, in the case of balls of different radii, the tangent balls are not the same anymore. It makes sense to consider the sums with coefficients that depend on the radii. As in \Cref{thm:linProg} above, let \(X=\{x_1,\ldots,x_m\}\subset S^{n-1}\) be the set of tangency points of a set of non-overlapping balls with the central unit ball. Let the radii of these balls be \(r_1,\ldots, r_m\), respectively. Let \(F:(0,\infty)^2\to\mathbb{R}\) be a function of two variables such that every finite principal submatrix is positive semidefinite, that is, for any finite set \(\{u_1,\ldots,u_m\}\), \(\left(F(u_i,u_j)\right)_{i,j}^m\succeq 0\). Then the analog of the Delsarte condition \(\sum\limits_{i,j=1}^m P_k^n(x_i\cdot x_j)\geq 0\) is the following lemma.

\begin{lemma}\label{lem:delsarte_SDP}
Under the conditions described above,
\[
\sum\limits_{i,j=1}^m F(r_i,r_j) P_k^n(x_i\cdot x_j)\geq 0.
\]
\end{lemma}

\begin{proof}
Follows from \Cref{prop:schoenberg} by the Schur product theorem.
\end{proof}

For a finite number of radii, this lemma is the essential part of the paper by de Laat, Oliveira, and Vallentin \cite{deLaat:2014}. The version in which \(F\) is a kernel on an infinite index set is used in the paper by Dostert, Kolpakov, and Oliveira \cite{Dostert:2022}.

One more ingredient of \Cref{thm:linProg} is the constraint \(x_i\cdot x_j\leq \frac 1 2\) that we do not have in the case of different radii. This inequality is replaced by a quick calculation that controls inner products based on the radii of the tangent balls.

\begin{lemma}\label{lem:ip}
If two non-overlapping balls of radii \(r\) and \(r'\) are tangent to the central unit ball at points \(x\) and \(x'\), respectively, then
\[
x\cdot x'\leq \frac{1+r+r'-rr'}{1+r+r'+rr'}.
\]
\end{lemma}

\begin{proof}
The centers of the three balls form a triangle with two side lengths \(1+r\), \(1+r'\) and the third side length \(\geq r+r'\). Finding the third side length by the law of cosines, we get the required inequality.
\end{proof}
In the rest of the paper we use \(\operatorname{ip}(r, r'):=\frac{1+r+r'-rr'}{1+r+r'+rr'}\).

\section{New approach via auxiliary functions}\label{sect:new}
While previous average kissing number upper bounds have come from \Cref{thm:densityKissingBound}, they rely on meaningful bounds on the cap packing density. This has led to significant improvements; however, relying on a volume argument presents limitations. Thus, we now present a new means for bounding average kissing numbers, which does not rely on the volumes of spherical caps. The new method is still based on the local configuration of one central ball and all balls in a packing that are tangent to it.
\begin{theorem} \label{thm:AvgKissBound}
Let \(g:(0,\infty) \rightarrow [0,\infty)\) be a function such that:
\begin{enumerate}
    \item For \(x \in (0,\infty)\), we have \(g(x)+g(\frac{1}{x}) \ge 2\), and
    \item For any configuration of non-overlapping \(n\)-dimensional balls of radii \(r_1,\ldots,r_k \in (0, \infty)\) tangent to an \(n\)-dimensional unit ball, we have \(g(r_1)+\ldots+g(r_k) \le M\).
\end{enumerate} 

Then, \(\kappa_n \le M\).  
\end{theorem}
\begin{proof}
Let \(P\) be any packing of \(N\) \(n\)-dimensional balls with radii \(r_1,\ldots, r_N\). Consider a ball of radius \(r_j\) and assume it is tangent to balls of radii \(r_{j,1}, \ldots, r_{j,k}\) from the packing \(P\). Note that, up to similarity, the configuration of balls of radii \(r_{j,1},\ldots, r_{j,k}\) tangent to the ball of radius \(r_j\) provides the same configuration as the balls of radii \(\frac{r_{j,1}}{r_j}, \ldots,\frac{r_{j,k}}{r_j}\) tangent to a unit ball. Then, by the hypothesis, \(g\left(\frac{r_{j,1}}{r_j}\right)+\ldots+g\left(\frac{r_{j,k}}{r_j}\right) \le M\).

Let $E$ be the set of edges in the contact graph of the packing, that is, $\{i,j\}\in E$ if the $i$-th and $j$-th balls are tangent to each other. We sum up these inequalities for all $j$, $1\leq j\leq N$. Then each edge $\{i,j\}$ contributes \(g\left(\frac {r_i}{r_j}\right) + g\left(\frac {r_j}{r_i}\right)\) to the sum and we get 

\[
\sum\limits_{\{i,j\}\in E} \left(g\left(\frac {r_i}{r_j}\right) + g\left(\frac {r_j}{r_i}\right)\right) \leq NM.
\]

It remains to notice that \(g\left(\frac {r_i}{r_j}\right) + g\left(\frac {r_j}{r_i}\right)\geq 2\) so $2|E|\leq NM$ and, subsequently \(\frac {2|E|}{N}\leq M\).

As this bound holds for any finite packing, we conclude that \(\kappa_n \le M\).
\end{proof}

\begin{remark}
    Note that this result immediately provides an alternative proof for the simple \(2 \tau_n\) bound on the average kissing number.  Namely, define:
    
    \begin{align*}
        g(x) =
\begin{cases}
0, & \text{if } x < 1, \\
1,  & \text{if } x = 1 \\
2, & \text{if } x>1
\end{cases}
    \end{align*}
 For any configuration of \(n\)-dimensional balls tangent to a unit ball, with radii \(r_1, \ldots, r_k\), we can consider \(g(r_1)+\ldots+g(r_k)\).  For any \(r_i<1\), we have \(g(r_i)=0\), and thus, at most \(\tau_n\) of these balls provide contribution towards this summation.  By definition, this contribution is at most \(2\). In other words, we have  \(g(r_1)+\ldots+g(r_k) \le 2 \tau_n\), implying that \(\kappa_n \le 2 \tau_n\).
\end{remark}

\begin{remark}
This approach strictly generalizes the shell approach of \Cref{thm:densityKissingBound}. Indeed, one can just define \(g(r)=\frac {A_{n,\rho(r)}} {A_{n,\rho(1)}}\). Then \(g(r)+g(1/r)\geq 2\) and \(g(r_1)+\ldots+g(r_k) \le \frac {\operatorname{dens_n(\rho)}}{A_{n,\rho(1)}}\).
\end{remark}

In order to find new upper bounds based on this approach, there are two main obstacles remaining. One is choosing a good auxiliary function, \(g\), subject to the hypotheses of \Cref{thm:AvgKissBound}. The second is that even having a function, \(g\), it is not immediately clear how we might rigorously bound the quantity \(g(r_1)+\ldots+g(r_k)\) for all possible configurations. Generalizing the approach in \cite{Dostert:2022} we estimate this quantity using \Cref{lem:delsarte_SDP}.

Before formulating the main results, we would like to analyze suitable auxiliary functions \(g\) and adjust them to simplify the program. If \(g(r_1)+\ldots+g(r_k) = M\) for some local configuration, which, for instance, is the case in dimensions 8 and 24, adding an arbitrary ball tangent to the unit ball and not overlapping with already present balls cannot possibly increase the sum. This means that \(g\) must be 0 on some interval \((0,x_0)\). Then the condition \(g(x)+g(1/x)\geq 2\) requires \(g(x)\geq 2\) for \(x\in (1/x_0,\infty)\). As we will see, exact equality for this interval is sufficient so we set \(g(x)=2\) for all \(x\in(1/x_0,\infty)\). In the active interval \([x_0, 1/x_0]\) we choose a more convenient parametrization of radii by setting \(c=\frac{1-x_0}{1+x_0}\) and taking \(r_c(u) =\frac{1+cu}{1-cu}\). This way \(r_c(-1)=x_0\) and \(r_c(1)=1/x_0\). The inequality \(g(x)+g(1/x)\geq 2\) in the new parametrization reads \(g(r_c(u))+g(r_c(-u))\geq 2\) for any \(u\in [-1,1]\). Finally, we define \(I_c(u,v)=\text{ip}(r_c(u), r_c(v))\) to describe the constraints from \Cref{lem:ip} in the new parametrization.

\begin{theorem} \label{thm:SDP}
    Let \(n \ge 3\) be an integer.  Choose \(0<x_0<1\), and set \(c=\frac{1-x_0}{1+x_0}\).  Let \(r_c:[-1,1] \rightarrow [x_0,1/x_0]\) be the increasing bijection defined by \(r_c(u) =\frac{1+cu}{1-cu}\), and then \(I_c(u,v)=\text{ip}(r_c(u), r_c(v))=\frac{1-c(u+v)-c^2uv}{2}\).  Let \(G:[-1,1] \rightarrow \mathbb{R}\) and \(\lambda>0\).  Fix an integer \(d>0\), and for \(k=0,\ldots,d\), let:
    \begin{align*}
        F_k:[-1,1]^2 \rightarrow \mathbb{R}.
    \end{align*}
    We write:
    \begin{align*}
        f(t, u, v)=\sum_{k=0}^{d}F_k(u,v) P_k^n (t) \hspace{1cm}\text{ for }t, u, v \in [-1,1]
    \end{align*}
    If:
    \begin{enumerate}
        \item\label{it:1} every principal submatrix of \(F_k\) is positive semidefinite for \(k=1,\ldots,d\)
        \item\label{it:2} every principal submatrix of \(F_0 -\frac{1}{\lambda}(G \otimes G)\) is positive semidefinite
        \item\label{it:3} \(f(t, u, v) \le 0\) whenever \(-1 \le t \le I_c(u,v)\)
        \item\label{it:4} \(f(1, u, u) \le G(u)\) for all \(u \in [-1,1]\)
        \item\label{it:5} \(G(u)+G(-u) \ge 2\) for all \(u \in [-1,1]\)
        \item\label{it:6} \(G(1) \ge 2\)
    \end{enumerate}
    Then \(\kappa_n \le \lambda\).
\end{theorem}
\begin{proof}
    We first want to show that \(G\) is non-negative on the domain.  Let \(u \in [-1,1]\).  Then, (\ref{it:1}) implies that \(F_k(u,u) \ge 0\) for \(k \ge 1\) and (\ref{it:2}) implies that \(F_0(u,u) -\frac{G(u)^2}{\lambda} \ge 0\).  Also, since \(P_k^n (1)=1\) for all \(k\),
    \begin{align*}
        f(1, u, u)=\sum_{k=0}^{d}F_k(u,u) \ge F_0 (u,u) \ge \frac{G(u)^2}{\lambda}
    \end{align*}
    By (\ref{it:4}), and since \(\lambda >0\), we have: \(0 \le \frac{G(u)^2}{\lambda} \le G(u)\) for all \(u \in [-1,1]\).
    
    We now define:
    \begin{align*}
        g(x) =
\begin{cases}
0, & \text{if } x < x_0, \\
\text{min}\{G(r_c^{-1}(x)), 2\},  & \text{if } x_0\le x \le 1/x_0 \\
2, & \text{if } x>1/x_0
\end{cases}
    \end{align*}
    We first show that \(g\) satisfies the condition of \Cref{thm:AvgKissBound}, that is, \(g(x)+g(1/x)\geq 2\). Let \(x \in (0, \infty)\).  By definition, \(g(x) \in [0,2]\).  Note that if \(x<x_0\) then \(1/x >1/x_0\).  So, \(g(x)+g(1/x) =2\).  If \(x>1/x_0\), then \(1/x<x_0\), so \(g(x)+g(1/x)=2\).  If \(x \in [x_0, 1/x_0]\), then \(1/x \in [x_0, 1/x_0]\) also.  So, \(r_c^{-1}(x)=u \in [-1,1]\).  We know that \(r_c^{-1}(1/x)=-u \in [-1,1]\) also.  So, \(g(x)+g(1/x) =\text{min}\{G(u), 2\} +\text{min} \{G(-u), 2\}\).  If both \(G(u), G(-u) \le 2\), then \(g(x)+g(1/x)=G(u)+G(-u) \ge 2\) (by (\ref{it:5})).  If one of \(G(u)\) or \(G(-u) \) is at least two, then we also have \(g(x)+g(1/x) \ge 2\).  So, for all \(x \in (0, \infty)\), we have \(g(x)+g(1/x) \ge 2\).  

    Now, consider a configuration, \(P\), of non-overlapping balls touching a central unit ball. Suppose the balls are of radii \(r_1, \ldots, r_m\). We want to bound \(g(r_1)+\ldots+g(r_m)\). This would lead to the hypotheses from \Cref{thm:AvgKissBound} being met, and thus would lead to a bound for the average kissing number. By definition of \(g\), any ball of radius less than \(x_0\) has no contribution to the summation, so, without loss of generality, we suppose that each \(r_i \ge x_0\).  

    Now, for each \(B \in P\), we call \(w_i \in S^{n-1}\) the point where \(B\) is tangent to \(S^{n-1}\). If \(B\) has radius \(r_i\) such that \(r_i \in [x_0, 1/x_0]\), then choose \(u_i \in [-1,1]\) such that \(r_c(u_i)=r_i\).  If \(r_i>1/x_0\), then choose \(u_i=1\) (basically, we replace the larger ball by a smaller one, such that the packing is still valid).  Now, the ordered pair, \((w_i, u_i)\) represents each ball in \(P\).  

    We now see that if \(r_i \in [x_0, 1/x_0]\), then \(g(r_i) =\text{min} \{G(u_i), 2\} \le G(u_i)\).  If \(r_i>1/x_0\), then \(g(r_i)=2 \le G(1)\) (by (\ref{it:6})).  So, we now have:
    \begin{align*}
        \sum_{i=1}^{m} g(r_i) \le \sum_{i=1}^{m} G(u_i)
    \end{align*}

    Define \(t_{ij} :=w_i \cdot w_j\), where \(t_{ii}=1\).  Consider:
    \begin{align*}
        f(t_{ij}, u_i, u_j)=\sum_{k=0}^{d}P_k^n (t_{ij})F_k(u_i, u_j)
    \end{align*}
    For \(k \ge 1\), 
    \[
    \sum\limits_{i,j=1}^m F_k(u_i,u_j) P_k^n(w_i\cdot w_j)\geq 0
    \]
    by Lemma \ref{lem:delsarte_SDP}.  

    Now consider that for \(k=0\), we have \(P_0^n \equiv 1\).  By (\ref{it:2}), we obtain that 
    \[
    \left(F_0(u_i, u_j) -\frac{G(u_i)G(u_j)}{\lambda}\right)_{1 \le i,j \le m}\succeq 0.
    \]
    So,
    \begin{align*}
        0 \le \mathbf{1}^T \left(F_0(u_i, u_j) -\frac{G(u_i)G(u_j)}{\lambda}\right)_{1 \le i,j \le m} \mathbf{1} =\sum_{1 \le i,j \le m}F_0(u_i, u_j) -\frac{1}{\lambda} \left(\sum_{i=1}^m G(u_i) \right)^2
    \end{align*}

    Then, we obtain:
    \begin{align*}
        \sum_{1 \le i,j \le m} f(t_{ij}, u_i, u_j) = \sum_{1 \le i,j \le m} \sum_{k=0}^{d} P_k^n(t_{ij})F_k(u_i, u_j) \ge \frac{1}{\lambda} \left(\sum_{i=1}^{m}G(u_i)\right)^2
    \end{align*}

    We now want to find an upper bound on this quantity.  First consider when \(i \ne j\).  We know that \(t_{ij} \le I_c(u_i, u_j)\), and we also know that \(t_{ij} \ge -1\).  By (\ref{it:3}), we then have
    \[
    f(t_{ij}, u_i, u_j) \le 0.
    \]
    When \(i=j\), then we have \(t_{ij}=1\), and by (\ref{it:4}), we have
    \[
    f(1, u_i, u_i) \le G(u_i).
    \]
    So:
    \begin{align*}
        \sum_{1 \le i,j \le m} f(t_{ij}, u_i, u_j) \le \sum_{i=1}^m G(u_i)
    \end{align*}
    We now have:
    \begin{align*}
        \frac{1}{\lambda} \left(\sum_{i=1}^{m}G(u_i)\right)^2 \le  \sum_{1 \le i,j \le m} f(t_{ij}, u_i, u_j) \le \sum_{i=1}^m G(u_i)
    \end{align*}
    Thus, this implies:
    \begin{align*}
        \sum_{i=1}^{m}G(u_i) \le \lambda
    \end{align*}
    We had previously shown that \(\sum_{i=1}^{m} g(r_i) \le \sum_{i=1}^{m} G(u_i)\). Now, by \Cref{thm:AvgKissBound}, we conclude that \(\kappa_n \le \lambda\).
\end{proof}

\begin{remark}
    We note that our choice of the increasing bijection \(r_c\) was deliberate. While \Cref{thm:SDP} could be stated with any increasing bijection, our specific choice simplifies condition (\ref{it:5}) and makes the constraining function \(I_c(u,v)\) polynomial in \(u\) and \(v\). 
\end{remark}

The following theorem is a modification of \Cref{thm:SDP} that is more convenient for calculations in some cases.

\begin{theorem}\label{thm:SDPCor}
    Let \(a:[-1,1] \rightarrow \mathbb{R}\), let \(F_0, \ldots, F_d:[-1,1]^2\to\mathbb{R}\), let \(M >0\), and let \(f, x_0, c\), \(r_c\) and \(I_c\) be defined as in \Cref{thm:SDP}.
    If:
    \begin{enumerate}
        \item\label{it:1cor} every principal submatrix of \(F_k\) is positive semidefinite for \(k=1,\ldots,d\)
        \item\label{it:2cor} every principal submatrix of \(F_0 -a \otimes a\) is positive semidefinite
        \item\label{it:3cor} \(f(t, u, v) \le 0\) whenever \(-1 \le t \le I_c(u,v)\)
        \item\label{it:4cor} \(f(1, u, u) \le M\) for all \(u \in [-1,1]\)
        \item\label{it:5cor} \(a(u) \ge 0\) for all \(u \in [-1,1]\)
        \item\label{it:6cor} \(a(u)^2+a(-u)^2 \ge 2\) for all \(u \in [-1,1]\)
        \item\label{it:7cor} \(a(1)^2 \ge 2\)
    \end{enumerate}
    Then \(\kappa_n \le M\).
    
\end{theorem}
\begin{proof}
    Define
    \[
    F_k^\ast(u,v) =\frac{a(u)a(v)}{M}F_k(u,v)
    \]
    for \(k=0,\ldots,d\), and
    \[
    f^\ast(t, u, v)=\frac{a(u)a(v)}{M}f(t, u, v).
    \]
    Define \(G(u)=a(u)^2\) for all \(u \in [-1,1]\). Set \(\lambda =M\).  

    Let \(\{u_1,\ldots, u_m \} \subset [-1,1]\). Define \(D\) to be the diagonal matrix with entries \(a(u_1),\ldots, a(u_m)\). Denote \[
    K_k=\bigl(F_k(u_i,u_j)\bigr)_{i,j=1}^{m},\quad
    K_k^\ast=\bigl(F_k(u_i,u_j)\bigr)_{i,j=1}^{m},\quad
    \mathbf a=(a(u_i))_{i=1}^{m},\quad
    \mathbf G=(G(u_i))_{i=1}^{m}.
    \]
    Now,
    \[
    K_k^*=\frac1M DK_kD.
    \]
    Because every principal submatrix of \(F_k\) is positive semidefinite, every principal submatrix of \(F_k^\ast\) is positive semidefinite too.  
    \[
    K_0^*-\frac1M\mathbf G\mathbf G^T=
    \frac1M D\bigl(K_0-\mathbf a\mathbf a^T\bigr)D.
    \]
    Since every principal submatrix of \(F_0-a \otimes a\) is positive semidefinite, every principal submatrix of \(F_0^\ast -\frac{1}{M}G \otimes G\) is positive semidefinite too.
    
    We see that for the relevant domain, \(f \le 0\). Then, for the relevant domain, we have \(f^\ast =\frac{a(u)a(v)}{M} f\le 0\), as \(a(u), a(v), M \ge 0\).  

    Now, consider
    \[
    G(u)-f^\ast(1, u, u)=\frac{a(u)^2}{M}(M-f(1, u, u)).
    \]
    By (\ref{it:4cor}), \(M-f(1, u, u) \ge 0\), and also \(a(u)^2, M \ge 0\).  Thus, \(G(u) \ge f^\ast(1, u, u)\) for all \(u \in [-1,1]\). Also, \(G(u)+G(-u)=a(u)^2+a(-u)^2 \ge 2\) by (\ref{it:6cor}).  And \(G(1)=a(1)^2 \ge 2\) by (\ref{it:7cor}). Therefore, all assumptions are met for \Cref{thm:SDP}, and thus we have that \(\kappa_n \le M\).  
    
    \end{proof}

    \begin{remark}
    We note that Theorems \ref{thm:SDP} and \ref{thm:SDPCor} are actually equivalent in the sense that they prove the same bound using the same function \(g\) in \Cref{thm:AvgKissBound}. In fact, we initially began with \Cref{thm:SDPCor} (rather, a slightly more general formulation) in our attempts to bound average kissing numbers. However, \Cref{thm:SDP} has a couple of advantages so we list it as the first, main result. A typical application of both theorems would search for functions \(G\) and \(a\) in the space of polynomials of a fixed degree. In this truncated case the connection \(G=a^2\) means that \Cref{thm:SDPCor} takes care only of polynomials \(G\) that are exact squares.
    Moreover, condition (\ref{it:6cor}) from \Cref{thm:SDPCor} is \(a(u)^2+a(-u)^2 \ge 2\) which cannot be transformed to an SDP constraint. If we want to optimize over \(a\), we need certain tricks to overcome this. Condition (\ref{it:5}) is \(G(u)+G(-u)\geq 2\) and can be easily interpreted as an SDP constraint.
    \end{remark}

    \begin{remark}
    Note that \Cref{thm:SDPCor} is a direct generalization of \cite[Theorem 3.1]{Dostert:2022} after certain reparametrizations and normalizations. If a function \(a(u)\) satisfies the required conditions of \cite[Theorem 3.1]{Dostert:2022}, most importantly, \(a(u)\geq A_{n,\rho}(r(u))^{1/2}\), one can easily modify it to \(a^*(u)=\frac{a(u)}{A_{n, \rho}(1)^{1/2}}\) satisfying the conditions of \Cref{thm:SDPCor}. We also note that there is a straightforward alternative proof of \Cref{thm:SDPCor} that follows the steps of \cite[Theorem 3.1]{Dostert:2022} combined with our \Cref{thm:AvgKissBound}. Of course our approach provides a lot more freedom in choosing auxiliary functions because they do not have to be tied to volumes of spherical caps at all.
    \end{remark}

    At the end of this section we discuss the limitations of Theorems \ref{thm:SDP} and \ref{thm:SDPCor}. In particular, one cannot possibly get an upper bound on average kissing numbers that is below the linear programming bound \(f(1)/f_0\) from the Delsarte method in \Cref{thm:linProg}. Indeed, we consider a restriction to configurations of only unit balls. In our parametrization it means that we consider \(u=v=0\) only and the univariate polynomial \(f(t)=f(t, 0,0)=\sum_{k} F_k(0,0)P_k^n(t)\). This polynomial satisfies all conditions of \Cref{thm:linProg}. Indeed, \(f_k=F_k(0,0) \ge 0\), \(f(t)=f(t,0,0) \le 0\) for \(-1 \le t \le I_c(0,0) = \tfrac12\). Since \(f(1)\leq G(0)\) and \(f_0=F_0(0,0) \ge G(0)^2/\lambda\), \Cref{thm:linProg} provides the bound
    \[
    \tau_n\leq \frac{f(1)}{f_0}\leq \frac {G(0)} {G(0)^2/\lambda} = \frac {\lambda} {G(0)} \leq \lambda
    \]
    because \(G(0) \ge 1\). In dimensions 2, 8, and 24, the Delsarte method provides sharp linear programming bounds for kissing numbers. If our approach produces the same exact bounds in these dimensions, then the exact solution must contain an optimal polynomial for the classical linear-programming bound as its unit radii slice.
    
    As we will see below, there are exact solutions of the average kissing number problem in dimensions 8 and 24. However, somewhat surprisingly, the method does not work in dimension 2. Moreover, \Cref{thm:AvgKissBound} cannot possibly produce an exact bound \(\kappa_2\leq 6\). This can be shown by iteratively estimating \(x_0\) for a function \(g\) until we reach a contradiction.

    \begin{proposition}\label{prop:6}
    There is no function \(g\) satisfying the conditions of \Cref{thm:AvgKissBound} for two-dimensional balls with \(M=6\).
    \end{proposition}

    \begin{proof}
    Assume such \(g\) exists. Then \(g(1)=1\) due to the kissing configuration. In this configuration we can add a disk of radius \(x_1=\frac {2\sqrt{3}} 3 - 1\) (the radii here and below are calculated by Descartes' circle theorem; see \cite{Lagarias:2002}). Therefore, \(g(x)=0\) on \((0,x_1]\). Then \(g(x)\geq 2\) on \([1/x_1,\infty )\). Note that we can take three non-overlapping disks of radius \(1/x_1\) tangent to the unit disk so \(3g(1/x_1)\) must be 6. Now for this configuration, we can take a disk of radius \(x_2=\frac {2\sqrt{3}+1} {11}\) tangent to the unit disk and not overlapping with three larger disks.  Therefore, \(g(x)=0\) on \((0,x_2]\). Consider now the following configuration of three disks of radii \(1/x_2=2\sqrt{3}-1\approx 2.46\) and one unit disk tangent to the unit disk. In a polar coordinate system with the origin at the center of the central disk, centers of larger disks are taken with polar angles \(\frac {5} {12} \pi, \pi, \frac {19} {12} \pi\) and the center of the unit disk is taken with polar angle 0. Using the calculation from \Cref{lem:ip}, we can see that the disks do not overlap. Their total contribution in terms of values of \(g\) is at least 7 so we get a contradiction. 
    \end{proof}
    
    \section{Initial formulation of SDP}\label{sect:SDP}
    In this section we describe the first results using our approach and the initial formulation of the semidefinite program based on \Cref{thm:SDPCor}.

    As our first attempt to improve upper bounds for average kissing numbers, we developed the following scheme. We first fix the cutoff value \(x_0\) and choose a function \(g\) on the interval \([x_0,1/x_0]\) such that \(g(x)+g(1/x)\geq 2\) and the function itself is "nice" in the sense that \(g(x)+g(1/x)\) is really close to 2 for all \(x\). The tradeoff for "niceness" is the fact that we cannot easily estimate the constant \(M\) in \Cref{thm:AvgKissBound}. In order to do that for the next step we find a polynomial function \(a\) approximating \(g^{1/2}\) from above and use \Cref{thm:SDPCor}. See the Master's thesis of Bockhaus \cite{Bockhaus:2026} for a full discussion on this method and its implementation. With this approach, we obtained the following new bounds on average kissing numbers: \(\kappa_3 \le 13.395, \kappa_4 \le 26.174, \kappa_5 \le 48.160, \kappa_6 \le 85.838, \kappa_7 \le 143.169, \kappa_8 \le 240.433, \kappa_9 \le 384.527, \kappa_{10} \le 608.127, \kappa_{11} \le 951.419\) and \(\kappa_{12} \le 1488.431\) \cite{Bockhaus:2026}.

    These results notably include \(\kappa_8 \le 240.433\).  The \(E_8\) lattice implies \(\kappa_8 \ge 240\). With a relatively small gap between lower and upper bounds, we conjectured that an exact solution could be obtained.  However, using an upper bound approximation for \(a\)  introduces error that prevents an exact solution.  Our next attempt at improving bounds involved choosing \(a\) directly. Let \(a(u)=\sum_{i=0}^N \alpha_i u^i\), where \(\alpha =(\alpha_0,\ldots, \alpha_N)\) determines \(a\). We first formulate the optimization problem described by \Cref{thm:SDPCor}, and then discuss the choices of an appropriate \(a\).

    Recall that we have kernels \(F_k:[-1,1]^2 \rightarrow \mathbb{R}\) and the function
    \[
    f(t,u,v) =\sum_{k=0}^{d} F_k(u,v)P_k^{n}(t).
    \]
    In practice, we parametrize the kernels \(F_k\) by matrices \(A_k\) and define:
\begin{align*}
    F_k(u,v) = \sum_{i,j =0}^{N} A_{k, ij} u^i v^j, \hspace{.5cm} \text{for all } u, v \in [-1,1].
\end{align*}

If \(A_k\) is a positive semidefinite \((N+1) \times(N+1)\) matrix, for all \(k=0,\ldots,d\), then all principal submatrices of \(F_k\) are positive semidefinite. Similarly, if \(A_0 -\alpha \alpha^T\) is positive semidefinite, then all principal submatrices of \(F_0-a \otimes a\) are positive semidefinite. Thus, we achieve the first two assumptions from \Cref{thm:SDPCor} by requiring two positive semidefinite conditions.

We still need to meet the third assumption: non-positivity of \(f\) on the relevant domain.  First, we see that \(f\) is a polynomial of \(u, v,\) and \(t\), and for \(-1 \le t \le I_c(u,v)\), we require non-positivity.  From this inequality, we see that we must have \(s_1(t, u, v)=t+1 \ge 0\) and \(s_2(t, u,v)=I_c(u,v) -t \ge 0\).  We want to ensure that \(u,v \in [-1,1]\) holds, and see this is true if and only if \(s_3(t, u, v)=(1-u^2)+(1-v^2) \ge0\) and \(s_4(t, u, v)=(1-u^2)(1-v^2) \ge 0\).  Thus, we can define the domain for which \(f\) needs to be non-positive as \(\{(t, u, v): s_i(t, u, v) \ge 0, i=1,2,3,4\}\).  

As we are trying to find an optimal \(f\) that is non-positive on this region, we may want to decrease the possible options, to reduce the size of the program.  We do this by introducing a sum-of-squares condition. Namely, we require sum-of-squares polynomials, \(q_1, q_2, q_3, q_4, q_5\) such that:
\begin{align} \label{eq:4.1}
f=-s_1q_1-s_2q_2-s_3q_3-s_4q_4-q_5
\end{align}
Thus, by construction, we will have that \(f\) is non-positive on our domain.  Note that the sum-of-squares condition is also useful, as saying that a polynomial, \(q(x)\), is a sum-of-squares is equivalent to saying that \(q(x)=m(x)^T Q m(x)\) where \(Q\) is a positive semidefinite matrix, and \(m(x)\) is a vector of monomials. This sufficient sum-of-squares condition can be represented by positive semidefinite constraints.  

Lastly, we consider the objective of the program: minimizing the maximum of \(f(1, u,u )\), as this will provide our upper bound on the average kissing number. We use the Markov-Luk\'{a}cs theorem \cite{Lukacs:1918}, \cite[Theorem 1.21.1]{Szego:1975} to obtain that this maximum is the same value as the minimum \(M\) such that:
\begin{align} \label{eq:4.2}
    M - f(1, u, u) =l_1(u)+(1-u^2)l_2(u),
\end{align}
where \(l_1\) and \(l_2\) are univariate sum-of-squares polynomials. Therefore, we now have our final problem as follows: 
{
\abovedisplayskip=0pt
\begin{align} 
    \text{minimize } &M  \label{eq:SDP} \\
    & f(t, u, v)=\sum_{k=0}^{d}\sum_{i,j =0}^{N} A_{k, ij} u^i v^j P_{k}^{n}(t) \notag \\
    &q_1, \ldots, q_5 \text{ are sum-of-squares polynomials satisfying } (\ref{eq:4.1}) \notag \\
    & l_1 \text{ and } l_2 \text{ are sum-of-squares polynomials satisfying }(\ref{eq:4.2}) \notag \\ 
    &A_0 -\alpha \alpha^T \text{ is positive semidefinite} \notag \\
    & A_k \in  \mathbb{R}^{(N +1) \times (N+1)} \text{ is positive semidefinite for }k=0,\ldots,d \notag
\end{align} }
Note that this problem is similar to that used by Dostert, Kolpakov, and Oliveira \cite{Dostert:2022}. The difference comes from \Cref{thm:AvgKissBound}, which provides more flexibility for the choice of the weight function, \(a\), and from the bijection \(r_c\). This approach improves the previous upper bounds in dimensions 3 through 30 and even solves the average kissing number problem in 8 and 24 dimensions.

\section{Exact solutions}\label{sect:exact}
In this section we explain how exact polynomial certificates in 8 and 24 dimensions were found. We want to mention that the choice of appropriate \(x_0\) and \(a\) was essentially a long trial-and-error process in which we used assistance from AI tools. There seems to be an interval of values \(x_0\) and a whole variety of auxiliary functions \(a\) which confirm the exact bound. For the exposition below we have chosen the least complex values.

\subsection{Choice of parameters}
We need to find an appropriate weight function \(a\) satisfying the hypotheses of \Cref{thm:SDPCor}. Therefore, we need \(a(u) \ge 0\) for all \(u \in [-1,1]\), \(a(1)^2 \ge 2\), and \(a(u)^2+a(-u)^2 \ge 2\) for all \(u \in [-1,1]\).  Also, as we wish to obtain exact bounds (meaning we achieve the linear programming bound), we need to assign balls of unit radius a weight of exactly one. This means we need \(a(0)=1\).  Looking for valid weight functions we consider polynomials of degree two in attempts to find a simpler solution.

The last condition \(a(0)=1\) gives us a polynomial of the form \(a(u)=1+a_1u+a_2u^2\).  Consider that we need \(a(u)^2+a(-u)^2 \ge 2\), that is, \(2(a_1^2+2a_2)u^2+2a_2^2u^4 \ge 0\). We would like the left-hand side to be as close to 0 as possible so we want to eliminate the quadratic term. Hence we choose \(a_2=-\frac{1}{2}a_1^2\). With this choice, \(a(u)^2+a(-u)^2-2=\frac{1}{2}a_1^4u^4 \ge 0\), for all choices of \(a_1\).  

We now consider the options for \(a_1\).  From \(a(1)^2 \ge 2\), we see that we must have \(1+a_1-\frac{1}{2}a_1^2 \ge \sqrt{2}\), which implies \(a_1 \ge 2-\sqrt{2}\).  As \(a(u) \ge 0\) for all \(u \in [-1,1]\), then \(a(-1) =1-a_1-\frac{1}{2}a_1^2 \ge 0 \), which implies that \(a_1 \le \sqrt{3}-1\). We now choose our \(a_1\) within this interval. We note that, interestingly, the choice is not unique (we were able to obtain exact solutions with several choices of \(a_1\)).  Choosing \(a_1=2/3\), we obtain \(a_2=-2/9\), for a rather simple polynomial of \(a(u)=1+\frac{2}{3}u-\frac{2}{9}u^2\).  This polynomial now satisfies conditions (\ref{it:5cor})-(\ref{it:7cor}) in \Cref{thm:SDPCor}.

With this choice, we now write the optimization problem as a semidefinite program. We use the Julia code adapted from the code of Dostert, Kolpakov, and Oliveira \cite{Dostert:2022} and the supporting code by Dostert, de Laat, and Moustrou \cite{DDM:2021}, making the changes dictated by our new weight function, \(a\), and the increasing bijection, \(r_c\). We now choose \(x_0=\frac{7}{10}\) and \(x_0=\frac{9}{10}\) for dimensions 8 and 24, respectively. Solving Problem (\ref{eq:SDP}) with the SDPA-GMP solver \cite{Nakata:2010} in Julia (using \(N=5\) for both, and \(d=10\) and \(d=14\) for 8 and 24 dimensions, respectively), we obtain numerical solutions with objective values agreeing with 240 and 196560 to the working precision.

\subsection{Rounding process}

We note that the raw output from the solver may not accurately meet all constraints from \Cref{thm:SDPCor} due to floating point arithmetic. To obtain a valid certificate, we now round the solutions found by the solver, so that each number in the solution is a rational number. For this process, we only consider the matrices, \(A_0,\ldots, A_d\). We round the entries of each matrix, enforcing positive semidefinite conditions directly. Since the bounds are exact, we have to trace carefully which inequalities actually become equalities in order to establish the exactness of the bound.

Rational recovery of numerical sum-of-squares decompositions was studied by Peyrl and Parrilo \cite{Peyrl:2008}. For semidefinite certificates arising in packing problems, including solutions that are not strictly feasible, see Dostert, de Laat, and Moustrou \cite{DDM:2021}. In the present setting, attaining the bounds in dimensions 8 and 24 forces the equality constraints described below.
We discuss the case of the kissing configuration from \(E_8\) in detail. The case of the Leech lattice is analogous, up to minor modifications that we also mention below.  

Let \(x_1,\ldots, x_{240}\) be the \(240\)  points of tangency between unit balls tangent to a central unit ball, in the \(E_8\) configuration. As before, \(t_{ij}\) represents \(x_i \cdot x_j\).  So, we know that \(x_i \cdot x_j \in \{-1, -\frac{1}{2}, 0, \frac{1}{2}, 1\}\).

We have \(F_k(0,0) \ge 0\), \(\sum_{i,j}P_{k}^n (t_{ij}) \ge 0\), \(P_0^n \equiv 1\), \(F_0(0,0) \ge a(0)^2\), and \(a(0)=1\).  Thus:
\begin{align*}
    \sum_{i, j}f(t_{ij}, 0, 0) =\sum_{k=0}^d F_k(0,0) \sum_{i,j}P_{k}^n (t_{ij}) \ge F_0(0,0) \tau_8^2 \ge a(0)^2 \tau_8^2 =\tau_8^2
\end{align*}
 We also have that for \(i \ne j\), then \(f(t_{ij},0,0) \le 0\), and that \(f(1, 0, 0) \le M\).  Then, for our configuration, we have:
 \begin{align*}
    \tau_8^2 \le \sum_{i, j}f(t_{ij}, 0, 0) \le \sum_{i} f(1, 0, 0) \le \tau_8 M 
\end{align*}

However, we have found an optimal bound of \(M=240\), which is the same as the kissing number in eight dimensions, \(\tau_8=240\), implying that all inequalities in the above manipulation are actually equalities.  Namely, this implies that \(f(t_{ij}, 0, 0)=0\) for all \(i \ne j\).  In particular, \(f(t_{ij}, 0, 0)=0\) for \(t_{ij} \in \{-1, -\frac{1}{2}, 0, \frac{1}{2}\}\).  This also implies \(f(1,0,0)=240\) and \(F_0(0,0)=1\).

We also require \(\partial_t f(-\frac12,0,0)=0\) and \(\partial_t f(0,0,0)=0\).  Note that the function \(t\mapsto f(t,0,0)\) is non-positive on \([-1,\frac{1}{2}]\) and vanishes at the interior points \(-\frac{1}{2}\) and \(0\), so these are local maxima. Since equality holds in the first inequality, \(F_k(0,0)\sum_{i,j}P_k^n(t_{ij})=0\) for every \(k\ge1\). We know that \(\sum_{i,j}P_k^n(t_{ij})=0\) for \(1\le k\le 10\), except at \(k=8\), where it is positive. Thus, we must impose \((A_8)_{00}=F_8(0,0)=0\). As \(A_8\) is positive semidefinite, its entire first row is zero. We also consider \((A_0-\alpha\alpha^T)_{00}\), and see that this is \((A_0)_{00}-\alpha_0^2=F_0(0,0)-1=0\). Similarly, as \(A_0-\alpha\alpha^T\) is positive semidefinite, its first row is all zeros.  In other words, the first row of \(A_0\) is \(\alpha^T\).

We now consider that \(\partial_u f(t,0,0)=\sum_k (A_k)_{01}P_k^n(t)\). For \(t\in\{-1,-\frac12,0\}\) we have \(t<\frac12(1-c)\le I_c(u,0)\) for all \(u\in[-1,1]\), so \(u\mapsto f(t,u,0)\) is non-positive on \([-1,1]\).  For these values of \(t\), we have \(f(t, 0,0)=0\).  We then impose \(\partial_u f(t,0,0)=0\) for \(t \in \{-1, -\frac{1}{2}, 0\}\). We have \(f(1, 0,0)=240\), and, by condition (\ref{it:4cor}), \(f(1, u, u) \le 240\).  So \(f(1,u,u)\) attains its maximum at \(u=0\).   Also, each \(F_k\) is symmetric, and thus \(0=\frac{d}{du}f(1,u,u)\big|_{u=0}=2\partial_u f(1,0,0)\).  So, we impose \(\partial_u f(1,0,0)=0\).  Finally, let \(\psi(u)=f(I_c(u,0),u,0)\). This path runs along the boundary \(t=I_c(u,v)\) of the region in condition (\ref{it:3cor}), so \(\psi(u)\le0\) for all \(u\in[-1,1]\), and \(\psi(0)=f(\frac12,0,0)=0\). Hence \(\psi\) has a local maximum at \(u=0\), and \(\psi'(0)=0\). Since \(I_c(u,0)=\frac12(1-cu)\), the chain rule gives \(\psi'(0)=\partial_u f(\frac12,0,0)-\frac c2\,\partial_t f(\frac12,0,0)\). We thus impose \(\partial_u f(\frac12,0,0)=\frac c2\,\partial_t f(\frac12,0,0)\).  

In a similar way, we can enforce the equalities imposed by the Leech lattice in 24 dimensions (where in this case we will have \(t_{ij} \in \{-1, -\frac{1}{2},-\frac{1}{4}, 0, \frac{1}{4}, \frac{1}{2}, 1\}\)). We impose analogous conditions for each of these points. Also, instead of \(A_8\) having the first row of zeros, we require \(A_{12}\) to have an entirely zero first row. 

\subsection{Verification}

After this rounding process, our solutions are given by the matrices \(A_0, \ldots, A_d\). In practice, we store the matrices \(
\widehat A_k = A_k/M\), where \(M=240\) in 8 dimensions and 196560 in 24 dimensions. Now we want to ensure that our solution truly provides a valid upper bound for the average kissing number. We do this by directly checking that our solution meets the hypotheses of \Cref{thm:SDPCor}. We meet the first two conditions by checking that \(\widehat A_1, \ldots, \widehat A_d, \widehat A_0-\frac{1}{M}\alpha \alpha^T\) are all positive semidefinite matrices (although this will already be true by virtue of our rounding process).  

While condition (\ref{it:3cor}) was initially enforced by sum-of-squares conditions, this is no longer guaranteed after rounding (namely, the sum-of-squares polynomials are removed from consideration, and only the matrices are rounded, resulting in the polynomial \(f\) being reconstructed from the rounded matrices). Thus, we verify this condition by Bernstein subdivision.  

We first map the region onto a box. Since \(1+I_c(u,v)>0\) on \([-1,1]^2\), the substitution
\(t=-1+s\,(1+I_c(u,v))\) maps \([0,1]\times[-1,1]^2\) onto \(\{(t,u,v):-1\le t\le I_c(u,v)\}\). Hence condition (\ref{it:3cor})
holds if and only if the polynomial
\[
h(s,u,v)=f\bigl(-1+s(1+I_c(u,v)),\,u,\,v\bigr)
\]
is non-positive on \([0,1]\times[-1,1]^2\). We subdivide this box and, on each sub-box, write \(h\) in the Bernstein basis. The tensor-product Bernstein basis is nonnegative and forms a partition of unity on each box. Consequently, the value of a polynomial is bounded above by its largest Bernstein coefficient, see Farouki \cite{Farouki:2012}. Thus, non-positive Bernstein coefficients certify non-positivity throughout the box. However, we note that this method cannot, in general, certify non-positivity on a box with a zero (where the zeros are determined by the \(E_8\) and Leech lattices, respectively). In this case, we then verify non-positivity on these boxes by using Taylor's Theorem to write the polynomial in terms of an expression with the Hessian.

Consider a box, \(B\), and a point \(p \in B\) such that \(h(p)=0\). Specifically, the relevant points \(p=(s_0, 0,0)\) satisfy \(s_0\in\{0,\tfrac13,\tfrac23,1\}\) in 8 dimensions and \(s_0\in\{0,\tfrac13,\tfrac12,\tfrac23,\tfrac56,1\}\) in 24 dimensions.
  
We now use the following results. If \(0<s_0<1\), \(\nabla h(p)=0\) and the Hessian is negative semidefinite for every point in the box, then \(h\) is non-positive on the box. If \(p\) lies on the face \(s=0\) of \(B\) (or \(s=1\)), \(\partial_sh\le0\) (or \(\partial_sh\ge0\)) on \(B\), \(\partial_uh(p)=\partial_vh(p)=0\), and the \(2 \times 2\) Hessian of \(h\) in \((u,v)\) is negative semidefinite at every point of the face of \(B\), then \(h\) is non-positive on the box. Note that several of these hypotheses hold because of the equalities we imposed in the rounding step.  Thus, with Bernstein subdivision and using these results for boxes with a zero, we may rigorously prove non-positivity on the region described in condition (\ref{it:3cor}).

Conditions (\ref{it:4cor}) and (\ref{it:5cor}) are also verified with Bernstein subdivision (although, we have ensured condition (\ref{it:5cor}) holds via construction). Condition (\ref{it:6cor}) was satisfied by construction of \(a\), and condition (\ref{it:7cor}) can easily be verified.  

After verifying that our solution satisfies each of the constraints from \Cref{thm:SDPCor}, we then conclude that we have a valid upper bound on the average kissing number. Namely, the approach described provides upper bounds of \(240\) and \(196560\) in 8 and 24 dimensions, respectively. The \(E_8\) lattice and Leech lattice show that these bounds are achieved as the limits of average contact numbers of their finite subpackings, and thus, we conclude that \(\kappa_8 =240\) and \(\kappa_{24}=196560\).  

As we noted earlier in this section, the solutions we found were not unique. We chose several weight functions, \(a\), that led to exact solutions. Even varying \(x_0\) within these weight functions, the solution often remained exact.

\section{Optimizing over auxiliary functions}\label{sect:opt}
In this section we explain how to obtain polynomial upper bound certificates for \(\kappa_n\) by optimizing over auxiliary functions \(G\) from \Cref{thm:SDP}. Before doing so, we need to mention that using small and simple polynomials, as we did for exact solutions in dimensions 8 and 24, in practice does not produce the best bounds in other dimensions. Therefore, it makes sense to extend the semidefinite program to include an optimization over \(G\). We begin by fixing \(N \ge 1\), and then saying \(G(u)=\sum_{i=0}^N \beta_i u^i\), where \(\beta =(\beta_0, \ldots, \beta_N)\) determines \(G\).  We use the same definition of the kernels \(F_k\), where we parametrize these by matrices \(A_k\).  

We see that requiring \(A_k\) to be a positive semidefinite \((N+1) \times(N+1)\) matrix, for \(k=1, \ldots, d\), we obtain that all principal submatrices of \(F_k\) are positive semidefinite.  

In a similar way, all principal submatrices of \(F_0 -\frac{G \otimes G}{\lambda}\) are positive semidefinite if and only if \(A_0 -\frac{\beta \beta^T}{\lambda}\succeq 0\). Let
\[
B= \begin{pmatrix}
A_0  & \beta \\
\beta^T & \lambda
\end{pmatrix}.
\]
Then \(A_0 -\frac{\beta \beta^T}{\lambda}\succeq 0\) if and only if \(B\succeq 0\). Therefore, condition (\ref{it:2}) is satisfied by requiring \(B\succeq 0\) to be a positive semidefinite matrix.  

To meet condition (\ref{it:3}), we require \Cref{eq:4.1}. For assumptions (\ref{it:4}) and (\ref{it:5}), we also employ sum-of-squares polynomials. We use the Markov-Luk\'{a}cs theorem to obtain that \(G(u) -f(1, u, u)\) is nonnegative if and only if it equals \(l_1(u)+(1-u^2)l_2(u)\) for sum-of-squares polynomials \(l_1\) and \(l_2\).  Similarly, \(G(u) +G(-u) -2\) is nonnegative if and only if it equals \(b_1(u)+(1-u^2)b_2(u)\) for sum-of-squares polynomials \(b_1(u)\) and \(b_2(u)\). Thus, we require sum-of-squares polynomials \(l_1, l_2, b_1, b_2 \) such that:
\begin{align}
    G(u)-f(1, u, u) =l_1(u)+(1-u^2)l_2(u)
\end{align}
\begin{align}
    G(u)+G(-u)-2 = b_1(u)+(1-u^2)b_2(u)
\end{align}

These conditions imply that \(f(1, u, u) \le G(u)\) and \(G(u)+G(-u) \ge 2\) for all \(u \in [-1,1]\), respectively. Assumption (\ref{it:6}) can be met by the linear condition:
\begin{align*}
    G(1)=\beta_0+\ldots+\beta_N \ge 2
\end{align*}

Therefore, we have the following optimization problem:
\begin{align}
    \text{minimize } &\lambda  \label{eq:SDPNEW} \\
    & f(t, u, v)=\sum_{k=0}^{d}\sum_{i,j =0}^{N} A_{k, ij} u^i v^j P_{k}^{n}(t), G(u)=\sum_{i=0}^{N}\beta_i u^i \notag \\
    & A_k \in  \mathbb{R}^{(N +1) \times (N+1)} \text{ is positive semidefinite for each }k=1,\ldots,d \notag \\
    &\begin{pmatrix}
A_0  & \beta \\
\beta^T & \lambda
\end{pmatrix} \text{ is positive semidefinite} \notag \\
    &q_1, \ldots, q_5 \text{ are sum-of-squares polynomials satisfying } (\ref{it:1}) \notag \\
    & l_1 \text{ and } l_2 \text{ are sum-of-squares polynomials satisfying }(\ref{it:4}) \notag \\ 
    & b_1 \text{ and } b_2 \text{ are sum-of-squares polynomials satisfying }(\ref{it:5}) \notag \\ 
    & \beta_0+\ldots+\beta_N \ge 2 \notag
\end{align}

\begin{remark}
    Note that there remains an omission from the optimization problem. Namely, we do not optimize over \(x_0\).  Unfortunately, we could not find a way to include \(x_0\) as a variable in the semidefinite program solver. We must choose \(x_0\in(0,1)\) explicitly in each dimension. However, there are still some observations we may make regarding a good choice for \(x_0\). As we mentioned above, \(x_0\) cannot be too close to zero, as if it were, then very small balls (that may fit between larger balls) would have a positive contribution to a summation of the weights. Also, \(x_0\) cannot be too close to \(1\), as if it were, then we may create an almost kissing configuration with balls of radius \(1/x_0\), in which case the bound would simply be \(2\tau_n\) or very close to it. With that being said, in general, we found that an "optimal" \(x_0\) gets larger (approaching one) as we increase the dimension. In practice, an appropriate \(x_0\) was found through a numerical search over simple rational choices.
\end{remark}

We now write Problem (\ref{eq:SDPNEW}) as a semidefinite program.  The ideas are very similar to those needed to turn Problem (\ref{eq:SDP}) into a semidefinite program, with further added conditions.  We now choose \(x_0\) to be a simple rational number, with the general observation that it gets closer to \(1\) as the dimension increases.  We also note that the degree of Gegenbauer polynomials tends to grow as the dimension goes higher. In general, we used the same (or one higher) degree, \(d\), as the one that provides the best linear programming upper bound for the kissing number.  Thus, as we reach higher dimensions, obtaining the best possible bounds becomes limited by computing power. While we use a kernel degree, \(N\), of five for lower dimensions, due to increased computing cost compared to relatively small improvements on the obtained bound, we lower this kernel degree to four for higher dimensions. With these choices of parameters, we now can solve the program.

After solving the optimization problem, we follow a similar process to that described in the discussion about solutions in 8 and 24 dimensions.  However, in the rational rounding process, we must now round matrices \(A_0,\ldots, A_d\) in addition to the bordered matrix corresponding to condition (\ref{it:2}). In this way, the weight function \(G\) is rounded as well. We again enforce positive semidefinite conditions by construction. In other dimensions, there are no exact lattice conditions to enforce and no exact bound that must be achieved. This makes non-positivity conditions slightly easier. Particularly, for condition (\ref{it:3}), we solve the semidefinite program by requiring a function, \(f\), that is less than some small \(-\varepsilon\).  Therefore, we assure that non-positivity holds after rounding. We do a similar process for conditions (\ref{it:4}) and (\ref{it:5}), where we implement this room for error with a small \(-\varepsilon\) on our given region. However, we still formally verify conditions (\ref{it:3})-(\ref{it:5}) via Bernstein subdivision.  With the room for error introduced, Bernstein subdivision was sufficient (there were no exact zeros, meaning that the process could successfully complete). Condition (\ref{it:6}) can be manually verified. After verifying that the six conditions from \Cref{thm:SDP} are met by the rounded solution, we obtain new upper bounds to the average kissing number in dimensions 3 through 30. 

\section*{Data}

Exact polynomial certificates for all dimensions 3 through 30 are available as standalone files attached to the paper.  The certificate files are also available via \url{https://github.com/markbockhaus/Exact_polynomial_certificates}.

\section*{Statement on the use of AI tools}

The main ideas and results of this paper are due to the authors. AI assistance was indispensable for writing and executing semidefinite programs that lead to the computational results and for the trial-and-error search process in finding concrete parameters for exact solutions. ChatGPT Sol 4.6 (OpenAI) was used in some computational stages of the project. Claude (Anthropic), using models Fable 5.1 and Opus 5.5, through Claude Code, was used by the authors to assist with the editing of existing code and writing of the new code for the semidefinite programs in Julia, the writing of code in Python and Julia for the rounding and verification procedures, and the running of the code. Both Claude and ChatGPT were used for reviewing the literature and for suggesting minimal edits to the text of the paper. The authors have reviewed all work done by AI and take full responsibility for the contents of the paper.

\section*{Acknowledgments}
 Mark Bockhaus was partially supported by the NSF S-STEM Program. Alexey Glazyrin was partially supported by the NSF grants DMS-2054536, DMS-2349063 and the Simons Foundation's Travel Support for Mathematicians program.

\bibliographystyle{amsplain}
\bibliography{references}

\end{document}